\documentclass{amsart}

\usepackage{geometry,graphicx,amssymb,amsmath,amsbsy,eucal,amsfonts,mathrsfs,amscd,bm,tcolorbox, subcaption, enumitem, color,mathtools,latexsym, tikz-cd}
  
\usepackage{stmaryrd} 
\usepackage{epstopdf,enumerate}
\usepackage[hidelinks]{hyperref}
\usepackage[all]{xy}
\newcommand{\Thh}{\mathcal{T}_h}

\numberwithin{equation}{section}
\allowdisplaybreaks[1]

\newtheorem{theorem}{Theorem}[section]
\newtheorem{lemma}[theorem]{Lemma}

\newtheorem{corollary}[theorem]{Corollary}
\newtheorem{proposition}[theorem]{Proposition}
\theoremstyle{definition}

\theoremstyle{remark}

\usepackage{mydef}

\begin{document}
\title[Elasticity Complex]{A vector-valued co-chain/chain complex associated to the elasticity complex}. 
%Connections between the Elasticity Complex and simplicial topology}

 \author{Jennifer Amador Pe\~{n}a}%
 \address{Mathematics, UC-Irvine, Irivine, CA}%
\email{jamadorp@uci.edu}%
\author{Sara de \'Angel}%
 \address{Applied Mathematics, Yale University, New Haven, CT }%
\email{sara.deangel@yale.edu}
\author{Fernanda Porras}%
 \address{Mathematics, Cal State Monterrey Bay, Monterrey, CA}%
 \email{fporras@csumb.edu}%

 \makeatletter
%\@namedef{subjclassname@2020}{\textup{2020} Mathematics Subject Classification}
%\makeatother
\subjclass[2020]{
%65N55;   %%  Multigrid methods; domain decomposition for boundary value problems involving PDEs;
%65F10;   %% Iterative numerical methods for linear systems
65N30;   %%  Finite element, Rayleigh-Ritz and Galerkin methods for boundary value problems involving PDEs;
58J10;   %%  Differential complexes [See also 35Nxx]; elliptic complexes
% 65N12;   %%  Stability and convergence of numerical methods for boundary value problems involving PDEs;
% 65N22;   %%  Numerical solution of discretized equations for boundary value problems involving PDEs;
% 65N15;   %%  Error bounds for boundary value problems involving PDEs
% 15A69;   %%  Multilinear algebra, tensor calculus
% 15A72;   %%  Vector and tensor algebra, theory of invariants [See also 13A50, 14L24]
% 74S05;   %%  Finite element methods applied to problems in solid mechanics
}

%\date{\today}
\begin{abstract}
In this report, we recall Christiansen's [2011, Num. Math.] discrete elasticity complex that uses spaces of polynomials and distributions. Using the degrees of freedom of those spaces, we write a complex using vector-valued chains and co-chains that is isomorphic to Christiansen's complex.  The operators of the new complex are algebraic, but use the geometry of the simplicial complex. Finally, we define a de Rham-like map from the smooth elasticity complex to this discrete complex.  \\

\begin{center}
{\bf Advisors:  Johnny Guzm\'an, Maurice Fabien, Parneet Gil, Elisabeth Rubio}
\end{center}
\end{abstract}

\maketitle

% \tableofcontents

\section{Introduction}

Consider a three-dimensional domain $\Omega \subset \R^3$ for which  the elasticity  complex is
\begin{equation}\label{elascomplex3d}
C^{\infty}(\overline{\Omega}) \otimes \mathbb{V}
\xrightarrow{\varepsilon}
C^{\infty}(\overline{\Omega}) \otimes \mathbb S
\xrightarrow{\inc}
C^{\infty}(\overline{\Omega}) \otimes \mathbb S
\xrightarrow{\div}
C^{\infty}(\overline{\Omega}) \otimes \mathbb{V}
\end{equation}
where $\varepsilon=\sym\grad$ is the linearized strain, and $\inc$ is the incompatibility
operator given by $\inc u=\curl (\curl u)^T$ for a matrix field $u$.    Here $\mathbb S$ is the space of symmetric $3 \times 3 $ matrices.  

In the last two decades, the elasticity complex has received enormous attention in the finite element exterior calculus (FEEC) community \cite{arnold2006finite,Arnold.D;Falk.R;Winther.R.2010a, ArnoldHu2021, ChristiansenGopalakrishnanGuzmanHu2024, ChenHuang2022}. Part of the reason why is that it is a canonical example of a derived complex using two de Rham complexes \cite{ArnoldHu2021}, which has had an explosion of activity. The other reason is that the spaces and operators have appeared in many applications such as linear elasticity, fluid problems, and general relativity. 

In this report, we aim to find a discrete {\it elasticity complex} using chains and co-chains (some vector-valued). We do this by using the work of Christiansen \cite{christiansen2011linearization} where he derives a discrete elasticity complex using spaces of polynomials and spaces of distributions. These spaces are connected with the differential operators that appear in the elasticity sequence.  What we do here is use the degrees of freedom of Christiansen's spaces to identify spaces of chains and co-chains and maps between them (see bottom complex in \eqref{discreteElasticity}). The maps between them will be algebraic maps that use some of the geometry of the mesh (e.g., the tangential vectors of edges and normal vectors of faces). The resulting complex will be reminiscent of the complex of co-chains with co-boundary operators. However, here both vector-valued co-chains and chains appear, and the maps have geometric information. Finally, we define a de Rham-like map that connects the elasticity complex with the new discrete complex (i.e., the maps $E_h^k$ in Section \ref{commutingdiscrete}).  

The report is structured as follows. In Section \ref{sec:notation} we develop the notation we will use. In Section \ref{derivation} we recall the derivation of the elasticity complex from two de Rham complexes.  Then, in Section \ref{deRham} we recall chains and co-chains and de Rham's maps.  In Section \ref{Christiansensection} we recall Christiansen's complex.  In Section \ref{finalsection} we identify the corresponding spaces of vector-valued chains and co-chains associated to the elasticity complex as well as their connecting maps (see the bottom complex of \eqref{Jcommute}). Finally, in the last Section \ref{commutingdiscrete} we write the associated commuting diagram with the de Rham-like maps.

\section{Preliminaries}\label{sec:notation}

\subsection{Tetrahedral Mesh}
We consider $\Omega\subset\mathbb R^3$ to be a bounded polyhedral domain with boundary
$\partial\Omega$. For an integer $k\ge0$, let $\mathbb P_k(D)$ denote the space of polynomials on
$D$ of total degree at most $k$. 

Let $\mathcal T_h$ be a tetrahedral mesh of $\Omega$, with mesh size $h$.  If
$T$ is a $d$-simplex, then $\Delta(T)$ denotes the set of all
subsimplices of $T$, and $\Delta_\ell(T)$ denotes the set of
$\ell$-dimensional subsimplices, $0\le \ell\le d$.  Finally, $\Delta_{\ell}$ denotes the set of
$\ell$-dimensional subsimplices of $\Th$.

\subsection{Differential operators}
\label{sec:DiffOper}
For a more comprehensive understanding, we must define differential operators that will appear in the elasticity complex. We may denote $\partial_i$ as the $i$th partial derivative.
The gradient operator, given as $\grad$, maps a scalar-valued function to a vector-valued function and is given by 

\[
   \grad \, u=
  \left[ {\begin{array}{c}
   \partial_1 u  \\
   \partial_2 u \\
\partial_3 u \\
  \end{array} } \right].
\]
When $F$ is a smooth vector-valued function
% \[
% F=
%   \left[ {\begin{array}{c}
%    F_1  \\
%    F_2 \\
%   F_3 \\
%   \end{array} } \right],
% \]
the curl operator is given by:
\[
\curl \, F =  \left[ {\begin{array}{c}
\partial_2 F_3-\partial_3 F_2 \\ 
-(\partial_1 F_3-\partial_3 F_1) \\
\partial_1 F_2- \partial_2 F_1 
\end{array} } \right].
\]
Finally, the divergence operator maps a vector-valued function to a scalar-valued function:
\begin{equation*}
\dive \, F =\partial_1 F_1+ \partial_2 F_2+ \partial_3 F_3.
\end{equation*}

We  consider the space of smooth functions on the closure of $\Om, \Omo$:
\begin{align*}
C^\infty(\Omo)&=\{ u: u \text{ and all its derivatives are continuous on } \Omo \}. 
\end{align*}
The de Rham complex, with smooth spaces, on $\Omega$ is given by
\begin{equation*}
C^{\infty} (\bar{\Omega}) \xrightarrow{\text{grad}} C^{\infty} (\bar{\Omega})\otimes \mathbb{V} \xrightarrow{\text{curl}} C^{\infty} (\bar{\Omega})\otimes \mathbb{V}  \xrightarrow{\text{div}} C^{\infty} (\bar{\Omega})
\end{equation*}

This is indeed a complex since  we have
\begin{equation}
 \curl\, \grad =0, \qquad \dive  \,\curl =0.
\end{equation}

\section{Derivation of the Elasticity Complex}\label{derivation}
\noindent Two de Rham complexes can be put together to derive the elasticity complex \cite{eastwood2000complex}. These two complexes are connected via algebraic maps that define skew-symmetric tensors. 
Let $\mathbb{V}=\mathbb{R}^3$  and let $\mathbb{M}$ be the space of real $3 \times 3$ matrices. We tensorize the de Rham sequence (\ref{sec:DiffOper}), which gives the following version of the de Rham complex: 
\begin{subequations}
\[
{\begin{tikzcd}[column sep=3em, row sep=3em]
{C^{\infty}(\bar {\Omega}) \otimes \mathbb{V}} \arrow[r, "\grad"] &
{C^{\infty}(\bar {\Omega}) \otimes \mathbb{M}} \arrow[r, "\curl"] &
{C^{\infty}(\bar {\Omega}) \otimes \mathbb{M}} \arrow[r, "\div"] &
{C^{\infty}(\bar {\Omega}) \otimes \mathbb{V}} 
\end{tikzcd}}
\]
\end{subequations}
where the operators now act row-wise. For example, the $\grad$ acting on a vector-valued function (written as a column vector) gives a matrix-valued function by taking the gradient of each component and placing it in the corresponding row.  By the properties of the differential operators, the range of the preceding map is contained in the kernel of the current map; thus the de Rham sequence is a complex. We can connect two of these complexes by using diagonal maps $-\text{mskw}, \Xi, \text{ and } 2\text{vskw}$ to obtain the following commuting diagram. 
\begin{equation}\label{doublecomplex}
{\begin{tikzcd}[column sep=3em, row sep=3em]
{C^{\infty}(\bar {\Omega}) \otimes \mathbb{V}} \arrow[r, "\grad"] &
{C^{\infty}(\bar {\Omega}) \otimes \mathbb{M}} \arrow[r, "\curl"] &
{C^{\infty}(\bar {\Omega}) \otimes \mathbb{M}} \arrow[r, "\div"] &
{C^{\infty}(\bar {\Omega}) \otimes \mathbb{V}} \\
{C^{\infty}(\bar {\Omega}) \otimes \mathbb{V}} \arrow[r, "\grad"] \arrow[ur, "-\text{mskw}"] &
{C^{\infty}(\bar {\Omega}) \otimes \mathbb{M}} \arrow[r, "\curl"] \arrow[ur, "\Xi"] &
{C^{\infty}(\bar {\Omega}) \otimes \mathbb{M}} \arrow[r, "\div"] \arrow[ur, "2\text{vskw}"] &
{C^{\infty}(\bar {\Omega}) \otimes \mathbb{V}}
\end{tikzcd}}
\end{equation}
The commuting properties can be written explicitly as 
\begin{subequations}
\begin{alignat}{1}
-\curl \text{mskw} =  &\Xi \grad  \\
\dive \Xi= &  2 \text{vskw} \curl. 
\end{alignat}
\end{subequations}
The first diagonal map, $\text{mskw}: \mathbb{V} \rightarrow \mathbb{M}$, associates each vector with a corresponding skew-symmetric matrix, and is defined as 
\[
 \text{mskw} \left[ {\begin{array}{c}
     u_1\\
     u_2\\
     u_3\\
  \end{array} } \right]
  =  \left[ {\begin{array}{ccc}
     0 & -u_3 & u_2 \\
     u_3 & 0 & -v_1 \\
     -u_2 & u_1 & 0 \\
  \end{array} } \right].
\]

\noindent The second diagonal map $\Xi: \mathbb{M} \rightarrow \mathbb{M}$ is a bijective algebraic transformation given as
\[
\Xi(\mathbb{M})=\mathbb{M}^T-\text{tr}(\mathbb{M)}\mathbb{I},
\]
and its inverse
\[
\Xi^{-1}(\mathbb{M})=\mathbb{M}^T-\frac{1}{2}\text{tr}(\mathbb{M)}\mathbb{I}.
\]
The third diagonal map, vskw: $\mathbb{M} \rightarrow \mathbb{V}$, is defined as
\[
\text{vskw} = \text{mskw}^{-1} \circ \text{skw}. 
\]

In order to derive a complex from the double complex \eqref{doublecomplex}, we will state an abstract result. Consider the following abstract complex 
\[
{\begin{tikzcd}[column sep=4.5em, row sep=3em]
Z_0 \arrow[r, "r_0"] &
Z_1 \arrow[r, "r_1"] &
Z_2 \arrow[r, "r_2"] &
Z_3 \\
V_0 \arrow[r, "t_0"] \arrow[ur, "s_0"] &
V_1 \arrow[r, "t_1"] \arrow[ur, "s_1"] &
V_2 \arrow[r, "t_2"] \arrow[ur, "s_2"] &
V_3
\end{tikzcd}}
\]

We use the following proposition to derive a single complex. This is a standard result, and the proof of this can be found in \cite{ChristiansenGopalakrishnanGuzmanHu2024}.
\begin{proposition}
Suppose that \\
(1) The sequences $\{Z_i\}$ and $\{V_i\}$  are complexes (e.g $r_{i+1} r_i=0$, $t_{i+1} t_i=0$) \\
(2) The diagram commutes (e.g. $r_1 s_0=s_1 t_0, r_2 s_1=s_2 t_1$),\\ 
(3)  $s_1$ is a bijection between $V_1$ and $Z_2$, 
Then, 
\[
\begin{bmatrix}
Z_0\\
V_0
\end{bmatrix}
\xrightarrow{\begin{bmatrix} r_0 & s_0 \end{bmatrix}}
Z_1
\xrightarrow{\,t_1\circ s_1^{-1}\circ r_1\,}
V_2
\xrightarrow{\begin{bmatrix} s_2 \\ t_2 \end{bmatrix}}
\begin{bmatrix}
Z_3\\
V_3
\end{bmatrix}.
\]
where 

\begin{equation*}
 \begin{bmatrix} r_0 & s_0 \end{bmatrix}   \begin{bmatrix}
z_0\\
v_0
\end{bmatrix} := r_0 z_0+s_0 v_0  \qquad  z_0 \in Z_0, v_0 \in V_0,
\end{equation*}
and 
\begin{equation*}
\begin{bmatrix}
s_2\\
t_2
\end{bmatrix} v_2 :=   
\begin{bmatrix}
s_2 v_2\\
t_2 v_2
\end{bmatrix}
\qquad  v_2 \in V_2.
\end{equation*}

\end{proposition}

Since the double complex \eqref{doublecomplex} and the corresponding maps satisfy the hypotheses of the above Proposition, we immediately obtain the following corollary.

\begin{corollary}
The following is a complex:
\begin{equation}\label{coll3.2}
\begin{bmatrix}
{C^{\infty}(\bar {\Omega}) \otimes \mathbb{V}} \\
{C^{\infty}(\bar {\Omega}) \otimes \mathbb{V}}
\end{bmatrix}
\xrightarrow{\begin{bmatrix} \grad & - \mskw \end{bmatrix}}
{C^{\infty}(\bar {\Omega}) \otimes \mathbb{M}}
\xrightarrow{\curl\Xi^{-1}\curl,}
{C^{\infty}(\bar {\Omega}) \otimes \mathbb{M}}
\xrightarrow{\begin{bmatrix} \vskw \\ \dive \end{bmatrix}}
\begin{bmatrix}
{C^{\infty}(\bar {\Omega}) \otimes \mathbb{V}}\\
{C^{\infty}(\bar {\Omega}) \otimes \mathbb{V}}
\end{bmatrix}.
\end{equation}
\end{corollary}

Before we derive the complex, we state an important identity that one can easily verify
\begin{equation}\label{716}
    \tr(\curl u)= 2 \dive (\text{vskw} u) \quad \text{for all } u \in {C^{\infty}(\bar {\Omega}) \otimes \mathbb{M}}.
\end{equation}
Consequently, if $u \in {C^{\infty}(\bar {\Omega}) \otimes \mathbb{S}}$  then  $\tr(\curl u)=0$ and we get the following 
\begin{equation}\label{717}
    \curl \Xi^{-1} \curl u = \inc u  \quad \text{for all } u \in {C^{\infty}(\bar {\Omega}) \otimes \mathbb{S}}.
\end{equation}

Now that we have this corollary, we can prove that indeed \eqref{elascomplex3d} is a complex.

\begin{theorem}
The sequence \eqref{elascomplex3d} is a complex. 
\end{theorem}

\begin{proof}
We need to show 
\begin{subequations}
\begin{alignat}{2}
    & \varepsilon u \in {C^{\infty}(\bar {\Omega}) \otimes \mathbb{S}} \qquad &&\forall u \in  {C^{\infty}(\bar {\Omega}) \otimes \mathbb{V}} \label{aux1.1}\\
     & \inc u \in {C^{\infty}(\bar {\Omega}) \otimes \mathbb{S}} \qquad  && \forall u \in  {C^{\infty}(\bar {\Omega}) \otimes \mathbb{S}}  \label{aux1.2}\\
     & \dive u \in {C^{\infty}(\bar {\Omega}) \otimes \mathbb{V}} \qquad  && \forall u \in  {C^{\infty}(\bar {\Omega}) \otimes \mathbb{S}}  \label{aux1.3}
\end{alignat}
   \end{subequations} 
and
\begin{subequations}
\begin{alignat}{2}
    & \inc \varepsilon u=0   \qquad &&\forall u \in  {C^{\infty}(\bar {\Omega}) \otimes \mathbb{V}} \label{aux2.1}\\
     & \div \inc u=0 \qquad  && \forall u \in  {C^{\infty}(\bar {\Omega}) \otimes \mathbb{S}}  \label{aux2.2} 
\end{alignat}
   \end{subequations} 
The statements \eqref{aux1.1} and \eqref{aux1.3} are trivial. Let's prove \eqref{aux1.2}. To this end,  let  $u  \in {C^{\infty}(\bar {\Omega}) \otimes \mathbb{S}}$. Then, since \eqref{coll3.2} is a complex, one has that $\text{vskw} \curl \Xi^{-1} \curl u=0$, which implies that $\text{vskw} \inc u =0$ if we use \eqref{717}. This implies that the skew-symmetric part of $\inc u$ is zero, so we have that $\inc u$ is symmetric.

To prove \eqref{aux2.1} we take a $u \in  {C^{\infty}(\bar {\Omega}) \otimes \mathbb{V}}$  and note that 
\begin{equation*}
 \begin{bmatrix} \grad & -\text{mskw} \end{bmatrix}   \begin{bmatrix}
u\\
\text{vskw} \grad u
\end{bmatrix} = \grad u -\text{mskw}\, \text{vskw} \grad u = \grad u -\text{skw} \grad u= \varepsilon u.
\end{equation*}
Since \eqref{coll3.2} is a complex we have that $\curl \Xi^{-1} \curl \varepsilon u=0$, which by \eqref{717}, gives $\inc \varepsilon u$. Finally, to prove \eqref{aux2.2}, let $u \in {C^{\infty}(\bar {\Omega}) \otimes \mathbb{S}}$. Then, since \eqref{coll3.2} is a complex, we have that $\dive \curl \Xi^{-1} \curl u=0$. The result follows if we use \eqref{717}.
\end{proof}

\section{Co-Chains and De Rham map}\label{deRham}
We aim to get an analogous de Rham map for the elasticity complex. To better appreciate it, we recall the de Rham map for the de Rham complex and the associated commuting diagram. In order to define the de Rham map, we need to define the co-chain complex, which in turn depends on the chain complex. 

Let $x_0, \ldots,  x_k$  be vertices of the triangulation $\Th$, then $\sigma=[x_0, \ldots, x_k]$ is an oriented $k$-simplex of $\Th$.  If one permutes any two vertices, then the resulting simplex will be $-\sigma$. Recall that $\Delta_k$ is the collection of $k$ sub-simplices of $\Th$. Then, we consider the space of $k$-chains 
\begin{equation}
\C_k=\{ \sum_{\sigma \in \Delta_k} a_{\sigma} \sigma: a_{\sigma} \in \R\}.
\end{equation}

The boundary map $\partial_k: \C_k \rightarrow \C_{k-1}$ is a linear map, which for $\sigma=[x_0, \ldots, x_k]$, is given by
\begin{equation*}
\partial_k \sigma= \sum_{j=0}^k (-1)^j[x_0, \ldots, \widehat{x_j}, \ldots x_k].
\end{equation*}
Here, $\widehat{\cdot}$ means omitting that entry. An important property is that $\partial_{k-1} \partial_{k}$ is the zero map. 
\begin{lemma}
 $\partial_{k-1} \partial_{k}$ is the zero map. 
\end{lemma}

Hence, we have the chain-complex

\begin{alignat}{4}\label{chaincomplex}
 \C_0
&&\stackrel{\partial_1}{\xleftarrow{\hspace*{0.5cm}}}\
 \C_1
&&\stackrel{\partial_2}{\xleftarrow{\hspace*{0.5cm}}}\
 \C_2
&&\stackrel{\partial_3}{\xleftarrow{\hspace*{0.5cm}}}\
\C_3
 \end{alignat}

We define the space of $k$ co-chains,  which we denote by $\C^k$, as follows:
\begin{equation*}
\C^k=\{ X: \C_k \rightarrow \R |  \,\, X \text{ is  a linear operator  } \} 
\end{equation*}
  For any $\sigma \in \Delta_k$, we define $  \sigma^*  \in \C^k$ as follows 
  $\langle \sigma^*, \tau \rangle =0$ if $\tau \in \Delta_k$ and $\tau$ is not a permutation of $\sigma$ and $\langle \sigma^*, \sigma \rangle=1$.  

Then, we can define the co-boundary operator $\mathsf{d}^k: \C^k \rightarrow \C^{k+1}$ in the usual form
\begin{equation}
\langle \mathsf{d}^k X, \tau \rangle= \langle X, \partial_{k+1} \tau\rangle, \qquad \forall \tau \in \C_{k+1}, X \in \C^k.
\end{equation}

If $\sigma=[x_0, \ldots, x_k]$, then it can be easily shown that:
\begin{equation*}
\mathsf{d}^k \sigma^*= \sum_{ x\in \Delta_0, [x, x_0, \ldots, x_k] \in \Delta_{k+1}} [x, x_0, \ldots, x_k]^*.  
\end{equation*}

We can easily show :
\begin{lemma}
 $\mathsf{d}^{k+1} \mathsf{d}^k$ is the zero map. 
\end{lemma}
\begin{proof}
We have for any $\tau \in \C_{k+2}$ and $X \in \C^k$
\begin{equation*}
    \langle \mathsf{d}^{k+1} \mathsf{d}^{k} X , \tau \rangle = \langle \mathsf{d}^k X, \partial_{k+2} \tau \rangle= \langle X,  \partial_{k+1} \partial_{k+2} \tau \rangle.
\end{equation*}
But, $\partial_{k+1} \partial_{k+2}$ is the zero map; hence $\mathsf{d}^{k+1} \mathsf{d}^{k}$  is the zero map. 
\end{proof}

Thus, we have a co-chain complex
\begin{subequations} \label{co-chain}
\[
{\begin{tikzcd}[column sep=3em, row sep=3em]
\C^{0} \arrow[r, "\mathsf{d}^{0}"] &
\C^{1} \arrow[r, "\mathsf{d}^{1}"] &
\C^{2} \arrow[r, "\mathsf{d}^{2}"] &
\C^{3}.
\end{tikzcd}}
\]
\end{subequations}

We can now define the de Rham maps $\dR^0: C^\infty(\Omo) \rightarrow \C^0$, $\dR^1: C^\infty(\Omo) \otimes \mathbb{V} \rightarrow \C^1$, $\dR^2: C^\infty(\Omo) \otimes \mathbb{V} \rightarrow \C^2$, $\dR^3: C^\infty(\Omo)  \rightarrow \C^3$ by
\begin{alignat*}{2}
\langle \dR^0 v, \sigma \rangle= &  v(\sigma), \quad  && \sigma \in \Delta_0, \\
\langle \dR^1 v, \sigma \rangle =& \int_{\sigma} v \cdot t,  \quad && \sigma \in \Delta_1, \\
\langle \dR^2 v, \sigma \rangle =&  \int_{\sigma} v \cdot n, \quad && \sigma \in \Delta_2, \\
 \langle \dR^3 v, \sigma \rangle =&  \int_{\sigma} v, \quad  && \sigma \in \Delta _3. 
 \end{alignat*}

It is very important to note that the following is a commuting diagram:
\begin{equation*}
  \begin{tikzcd}[ampersand replacement=\&,column sep=4em]
{C^{\infty}(\bar {\Omega})}
    \ar{d}{\dR^0}
    \arrow{r}{\grad}
    \& [0.5em] {C^{\infty}(\bar {\Omega})} \otimes \mathbb{V}
    \arrow{r}{\curl}
    \ar{d}{\dR^1}
    \&[1em]
 {C^{\infty}(\bar {\Omega})} \otimes \mathbb{V}
    \arrow{r}{\dive}
    \ar{d}{\dR^2}
    \&[1em]
{C^{\infty}(\bar {\Omega})}
    \ar{d}{\dR^3}
    \\
    \C^0
    \arrow{r}{\mathsf{d}^0}
    \&[0.3em]
   \C^1
   \arrow{r}{\mathsf{d}^1}
    \&[1em]
    \C^2
    \arrow{r}{\mathsf{d}^2}
    \&[1em] \C^3
  \end{tikzcd}
\end{equation*}

\section{Christiansen's discrete elasticity complex and commuting diagrams}\label{Christiansensection}
Christiansen \cite{christiansen2011linearization} gave a discrete elasticity complex such that the first two spaces are spaces of piecewise polynomials and the last two spaces are finite-dimensional spaces of distributions. We will follow the convention of Christiansen \cite{christiansen2011linearization} to now assume that $\Omega$ is the three-dimensional flat torus. That is, we start with the unit cube, and we abstractly glue opposite faces to each other. When we construct the mesh for the unit cube, we need to ensure that when we glue opposite faces that faces of the tetrahedra match. This will then give a valid simplicial complex of the flat torus. Christiansen did this (as we are) to avoid complications near the boundary of $\Omega$ when defining the distributional spaces. However, three-dimensional domains $\Omega$ with boundary have been considered in \cite{christiansen2026regge}. Now we recall Christiansen's discrete spaces. Note that now we do not need $\overline{\Omega}$  since we are working with a domain without boundary.

The first space is simply the space of vector-valued linear, Lagrange elements:
\begin{equation*}
 X_h^0 = \{v \in C(\Omega) \otimes \mathbb{V} : v|_T \in P^1(T) \otimes \mathbb{V}, \forall T \in \Thh \},   
\end{equation*}
where $P^k(T)$ denotes the space of polynomials of degree less than or equal to $k$ defined on $T$.
%% I have noticed that earlier when discussing the tetrahedra mesh, we defined a space of polynomials in a different way... should I change them to be the same? - Jennifer- I don't see where we defined space of polynomials before. 
The second space is given by:
\begin{equation*}
X^1_h = \{m \in L^2(\Om) \otimes \mathbb{S}: m|_T \in P^0(T) \otimes  \mathbb{S}, \forall T \in \Thh, \int_e t_e^T m \ t_e  \text{ is single-valued } \forall e \in \Delta_1 \}.    
\end{equation*}
Note that only tangential-tangential continuity across edges is imposed. 

In order to define the third space, we need the distribution $\delta_e: C^{\infty} (\Omega) \rightarrow \R$ for any $e \in \Delta_1$, which is given by
\begin{equation*}
\langle\delta_e, v \rangle = \int_e v.
\end{equation*}

Then, the third space is given by:
%Need to define t_e as the tangent vector to the edge
\begin{equation*}
X^2_h = \{\sum_e a_et_e t_e^T\delta_e: a_e \in \mathbb{R}\}.
\end{equation*}
An element $\sum_e a_et_e t_e^T\delta_e: C^{\infty}(\Omega) \otimes \mathbb{S} \rightarrow \R $ is defined by
\begin{equation*}
   \langle\sum_e a_et_e t_e^T\delta_e, m\rangle: = \sum_e a_e \int_e t_e^T mt_e .
\end{equation*}
 Indeed, we get 
\begin{equation*}
\langle\sum_e a_et_e t_e^T\delta_e, m\rangle = \sum_e a_e\langle \delta_e, t_e t_e^T:m\rangle=\sum_e a_e\langle \delta_e, t_e^T mt_e\rangle= \sum_e a_e \int_e t_e^T mt_e.
\end{equation*}

In order to define $X_h^3$, we need the following distribution for any vertex $y$, $\delta_y: C^{\infty} (\Omega) \rightarrow \R$ 
\begin{equation*}
    \langle \delta_y, v \rangle = v(y).
\end{equation*}

The final space is defined by:
\begin{equation*}
X^3_h = \{\sum_y a_y\delta_y: a_y \in \mathbb{V}\}.
\end{equation*}
% can change R^3 to V
An element $\sum_y a_y\delta_y: C^{\infty}(\Om) \otimes \mathbb{V} \rightarrow \R $ is given by
\begin{equation*}
   \langle \sum_y a_y\delta_y, u\rangle: = \sum_y a_y \cdot u(y). 
\end{equation*}

Christiansen's discrete complex can be written as 
\begin{equation}\label{ChristiansenComplex}
X_h^0
\xrightarrow{\varepsilon}
X_h^1
\xrightarrow{\inc}
X_h^2
\xrightarrow{\dive}
X_h^3
\end{equation}
% Did we define RM? -Jennifer. I removed it-Johnny
One interesting aspect is that the $\inc$ operator maps $X_h^1$ into $X_h^2$.

Given $u \in X_h^1$, we define the distribution $\inc u$ in the standard way
\begin{alignat}{1}
    \langle \inc u, v \rangle = (u, \inc v)  \qquad \forall v \in C^{\infty}(\Om) \otimes \mathbb{S}. \label{weakinc}
\end{alignat}
Note that we are using here that $\Omega$ does not have a boundary and is compact.

Christiansen proved the following result, which describes the image of the $\inc$ operator acting on an element of $X_h^1$. 
\begin{proposition}
    Let $u \in X_h^1$ then 
    \begin{alignat*}{1}
        \inc u= \sum_{e \in \Delta_1} \llbracket u \rrbracket_e t_e t_e^T \delta_e,  
    \end{alignat*}
where 
\begin{equation*}
\llbracket u \rrbracket_e= \sum_{F \in \Delta_2, e \in \Delta_1(F)} m_{eF}^T \llbracket u \rrbracket_{eF} n_{eF}.     
\end{equation*}
Here $m_{eF}$ is tangent to $F$ perpendicular to $t_e$ that points into $F$. The vector $n_{eF}$ is perpendicular to $F$ such that $(m_{eF}, n_{eF}, t_e)$ satisfies the right-hand rule: 
\begin{alignat*}{1}
 t_e \times m_{eF}=& n_{eF}, \\ 
 t_e \times n_{eF}=& -m_{eF}. 
\end{alignat*}
Finally, if $T_1, T_2$ are the two tetrahedra that share the face $F$ \big(e.g. $F \in \Delta_2(T_i)$ ($i=1,2$) \big) and if $n_{eF}$  points into $T_1$  (e.g. out of $T_2$) then  
\begin{equation*}
   \llbracket u \rrbracket_{eF} = u|_{T_1}- u|_{T_2}.
\end{equation*}
% is it u|{T_2}? - Jennifer-- Yes, you are right! I changed it now
\end{proposition}

Christiansen also proved the following result. 
\begin{equation*}
\end{equation*}
\begin{proposition}
Let $e =[x, z] \in \Delta_1$ then 
    \begin{alignat}{1}
        \dive (t_e t_e^T \delta_e) = t_e (\delta_z-\delta_x). \label{divXh2}  
    \end{alignat}
\end{proposition}
\subsection{Degrees of freedom of $X_h^k$}
In order to discuss Christiansen's interpolants and the commuting diagram, we need to discuss the degrees of freedom of the spaces $X_h^k$.

For the first space  $X_h^0$, we have the following degrees of freedom: \\

An element $u \in X_h^0$ is uniquely determined by
\begin{equation*}
    u(y) \quad \forall y \in \Delta_0. 
\end{equation*}

For the second space  $X_h^1$, we  the following degrees of freedom: \\

An element $m \in X_h^1$ is uniquely determined by
\begin{equation*}
    \int_e t_e^T m t_e  \quad \forall e \in \Delta_1. 
\end{equation*}

For the third space $X_h^2$, we have the following degrees of freedom: \\

An element $w \in X_h^2$ is uniquely determined by
\begin{equation*}
    \langle w, m \rangle \quad \forall  m \in X_h^1. 
\end{equation*}

Finally, for the fourth space $X_h^3$, one has the following degrees of freedom. 

An element $ v \in X_h^3$ is uniquely determined by
\begin{equation*}
    \langle v, u\rangle \quad \forall u \in X_h^0.
\end{equation*}
\subsection{Commuting interpolants}
Christiansen wrote the following interpolants that are induced by the above degrees of freedom. 

The first one   $I_h^0 : C^{\infty}(\Omega)\otimes \mathbb{V} \rightarrow X_h^0$ is given by
\begin{equation*}
   I_h^0 u(y)=u(y) \qquad \forall y \in \Delta_0.
\end{equation*}

The second one   $I_h^1 : C^{\infty}(\Omega)\otimes \mathbb{S} \rightarrow X_h^1$ is given by
\begin{equation}\label{Ih1}
   \int_e t_e^T  (I_h^1 m) t_e= \int_e t_e^T  m t_e \qquad \forall e \in \Delta_1.
\end{equation}

The third one   $I_h^2 : C^{\infty}(\Omega)\otimes \mathbb{S} \rightarrow X_h^2$ is given by
\begin{equation}
\langle I_h^2 w, m \rangle = (w, m)   \qquad \forall m \in X_h^1,  \label{Ih2}
\end{equation}
where $(w,m)=\int_{\Om} w: m$ is the $L^2$ inner-product on matrix fields.

The last one   $I_h^3 : C^{\infty}(\Omega)\otimes \mathbb{V} \rightarrow X_h^3$ is given by
\begin{equation*}
\langle I_h^3 v, u \rangle = (v, u)   \qquad \forall u \in X_h^0. 
\end{equation*}
where $(v,u)=\int_{\Om} v \cdot u$ is the $L^2$ inner-product on vector fields.

Let us try to understand $I_h^2$ more.  Since $I_h^2 w$ has the form 
\begin{alignat*}{1}
I_h^2 w=\sum_{e} a_e^w t_e t_e^T \delta_e,
\end{alignat*}
for constants  $a_e^w$, we see that
%% would it be valid to say that the a_e^w are constants specific to the matrix w? - Jennifer- Good point but I think it is implicit, so my reocmmendation should be to leave as is. 
\begin{equation*}
\langle I_h^2 w, m \rangle = \sum_e a_e^w \int_e t_e^T m t_e. 
\end{equation*}
Hence, we need to find the constants $\{a_e^w\}$ that satisfy:
\begin{equation*}
\sum_e a_e^w \int_e t_e^T m t_e= (w, m)   \qquad \forall m \in X_h^1. 
\end{equation*}
To obtain a square linear system for the constants $\{a_e^w\}$ we simply need to satisfy the above for a basis of $X_h^1$. We now describe a particularly useful basis for $X_h^1$. 

Define $\psi_e \in X_h^1$ to be the unique member of $X_h^1$ that satisfies 
\begin{equation}\label{me}
\int_g t_g^T \psi_e t_g =
\begin{cases}
1  \quad \text{ if } g=e, g \in \Delta_1 \\
0  \quad \text{ if } g \neq \pm e, g \in \Delta_1.
\end{cases}
\end{equation} 

Thus, $\{a_e^w\}$ must solve 
\begin{equation*}
\sum_e a_e^w \int_e t_e^T \psi_g t_e= (w, \psi_g)   \qquad g \in \Delta_1. 
\end{equation*}
This gives a square linear system. In fact, this gives us 
\begin{equation*}
    a_g^w= (w, \psi_g),
\end{equation*}

\begin{alignat*}{1}
I_h^2 w=\sum_{e} (w,\psi_e) t_e t_e^T \delta_e.
\end{alignat*}

Before moving on to commuting diagrams, we state an important result.  Although the proof did not appear in Christiansen \cite{christiansen2011linearization} one can prove this by a density argument (personal communication with Snorre Christiansen).
\begin{lemma}
It holds, 
\begin{equation}\label{incinc}
  \langle \inc u, v \rangle= \langle \inc v, u \rangle \qquad \text{for all } u, v \in X_h^1.  
\end{equation}
\end{lemma}

We can now state Christiansen's commuting diagram. 
\begin{equation}\label{Icommute}
\begin{tikzcd}[column sep=3em, row sep=3em]
C^{\infty}(\Omega)\otimes \mathbb{V} \arrow[r, "\varepsilon"] \arrow[d, "I_h^0"] &
C^{\infty}(\Omega)\otimes \mathbb{S} \arrow[r, "\inc"] \arrow[d, "I_h^1"] &
C^{\infty}(\Omega)\otimes \mathbb{S} \arrow[r, "\dive"] \arrow[d, "I_h^2"] &
C^{\infty}(\Omega)\otimes \mathbb{V} \arrow[d, "I_h^3"] \\
X_h^0 \arrow[r, "\varepsilon"] &
X_h^1 \arrow[r, "\inc"] &
X_h^2 \arrow[r, "\dive"] &
X_h^3
\end{tikzcd}
\end{equation}
We state this as the following theorem.
\begin{theorem}
The following commuting properties hold
\begin{subequations}
\begin{alignat}{1}
I_h^1 \varepsilon=  &\varepsilon I_h^0, \label{comm_1} \\ 
I_h^2 \inc = & \inc I_h^1,   \label{comm_2}\\
 I_h^3 \dive =& \dive I_h^2.  \label{comm_3}
\end{alignat}
\end{subequations}
\end{theorem}

\begin{proof}
We first prove \eqref{comm_1}. Let $u \in C^{\infty}(\Omega) \otimes \mathbb{V}$.
We want to show that $I_h^1 \varepsilon u= \varepsilon I_h^0 u \in X_h^1$. It is enough to show that 
\begin{equation*}
\int_e t_e^T(I_h^1 \varepsilon u)t_e = \int_e t_e^T(\varepsilon I_h^0 u)t_e, \quad  \forall e \in \Delta_1.    
\end{equation*}

Set $e$ to be the directed edge $e = [x_0 \ x_1]$.
Note that for any $n \times n$ matrix $M$ and $n \times 1$ vector $z$,
\begin{equation}\label{zMz=zsym(M)z}
    z^T M z = z^T \text{sym}(M) z.
\end{equation}

Starting with the left-hand side, by the definition of $I_h^1$,
\begin{align}
    \int_e t_e^T(I_h^1 \varepsilon u)t_e = & \int_e t_e^T(\varepsilon u)t_e \notag \\
    = & \int_e t_e^T(\grad u)t_e \label{comm_1_LHS}  
\end{align}
where we used $t_e^T(\grad u)t_e= t_e^T(\varepsilon u)t_e$.

Now, for the right-hand side, we again use \eqref{zMz=zsym(M)z} to write 
\begin{alignat}{1}
\int_e t_e^T(\varepsilon I_h^0 u)t_e =& \int_e t_e^T \grad (I_h^0 u) t_e \notag \\ 
 = & \int_e t_e^T \partial_{t_e} (I_h^0 u) \notag \\
 = & \int_e t_e \cdot \partial_{t_e} (I_h^0 u) \notag\\
 = & \int_e \partial_{t_e} (I_h^0 u \cdot t_e) \notag
 \end{alignat}.
 
 Using the Fundamental Theorem of Calculus and the defintion of $I_h^0$  we have  
 \begin{alignat}{1}
 \int_e t_e^T(\varepsilon I_h^0 u)t_e=\int_e \partial_{t_e} (I_h^0 u \cdot t_e)= & (I_h^0 u)(x_1) \cdot t_e - (I_h^0 u)(x_0) \cdot t_e \notag \\
 = & u(x_1) \cdot t_e - u(x_0) \cdot t_e \notag \\
 =& \int_e \partial_{t_e} (u \cdot t_e) \notag \\
 = & \int_e t_e^T(\grad u)t_e \label{comm_1_RHS}.
 \end{alignat}
Combining \eqref{comm_1_LHS} and  \eqref{comm_1_RHS} we arrive at the result. 

Now we prove \eqref{comm_2}.

Let $u \in C^{\infty}(\Omega)\otimes \mathbb{S}$  then we need to show

\begin{alignat*}{1}
 \langle I_h^2 \inc u, v \rangle  = &  \langle  \inc I_h^1 u, v \rangle    \qquad \text{ for all } v \in X_h^1.
\end{alignat*}
We work with the left-hand side 
\begin{subequations}
\begin{alignat}{2}
 \langle I_h^2 \inc u, v \rangle  = &  ( \inc u, v)  \qquad  &&\text{ by } \eqref{Ih2}, \\
 =& \langle \inc v, u \rangle  \qquad  &&\text{ by } \eqref{weakinc}, \\
 =& \langle \inc v, I_h^1 u \rangle \qquad &&\text{ by } \eqref{Ih1} \\
 =& \langle \inc I_h^1 u, v \rangle. \qquad &&\text{ by } \eqref{incinc}
\end{alignat}
\end{subequations}
\end{proof}

\section{Co-chains  and chains associated to the elasticity complex}\label{finalsection}

In this section, we obtain an isomorphism between Christiansen's spaces and a sequence of co-chains and chains (some vector-valued). This will allow us to define a de Rham-type map for the elasticity complex, which we do in the next section.

We seek to find a bottom complex such that the diagram commutes. We first define the vertical maps which will be isomorphisms. Then, we define the horizontal maps to satisfy the commutativity.
\begin{equation}\label{Jcommute}
\begin{tikzcd}[column sep=3em, row sep=3em]
 X_h^0 \arrow[r, "\varepsilon"] \arrow[d, "J_h^0"] & X_h^1 \arrow[r, "\inc"] \arrow[d, "J_h^1"] &
X_h^2  \arrow[r, "\dive"] \arrow[d, "J_h^2"] & X_h^3 \arrow[d, "J_h^3"] \\
\C^0 \otimes \mathbb{V} \arrow[r, "\mathsf{r}_0"] &
\C^1 \arrow[r, "\mathsf{r}_1"] &
\C_1  \arrow[r, " \mathsf{r}_2"] &
\C_0 \otimes \mathbb{V} 
\end{tikzcd}
\end{equation}

Here the space $\C^0 \otimes \mathbb{V}$ is the space of vector-valued zero co-chains. An arbitrary element in $\C^0 \otimes \mathbb{V}$  takes the form  $\sum_{y \in \Delta_0} c_y y^*$  where $c_y \in \mathbb{V}$. The space $\C_0 \otimes \mathbb{V}$ is the space of vector-valued zero chains. An arbitrary element $\C_0 \otimes \mathbb{V}$  takes the form  $\sum_{z \in \Delta_0} b_z \, z$ where $b_z \in \mathbb{V}$.  We define the action 
\begin{alignat*}{1}
\langle \sum_{y \in \Delta_0} c_y y^* , \sum_{z \in \Delta_0} b_z z \rangle = \sum_{y \in \Delta_0} c_y \cdot b_y.    
\end{alignat*}

The space $\C^1, \C_1$ are the space of $1$ co-chains and $1$ chains, respectively.

\subsection{The vertical maps $J_h^k$}
We can now define the vertical maps. 
\subsubsection{ The map $J_h^0: X_h^0 \rightarrow \C^0 \otimes \mathbb{V}$} 
For any $u \in X_h^0$, define
\begin{equation*}
    \langle J_h^0 u,  L  y \rangle = L \cdot u(y)    \qquad \text{ for } y \in \Delta_0, L \in \mathbb{V}.
\end{equation*}

\subsubsection{ The map $J_h^1: X_h^1 \rightarrow \C^1$} 
For any $u \in X_h^1$, define
\begin{equation*}
    \langle J_h^1 u, e \rangle=  \int_e t_e^T u t_e  \qquad \text{ for } e \in \Delta_1.
\end{equation*}

\subsubsection{ The map $J_h^2: X_h^2 \rightarrow \C_1$} 
For any $w \in X_h^2$, define
\begin{equation*}
    \langle e^*, J_h^2 w \rangle =  \langle w , \psi_e \rangle  \qquad \text{ for } e \in \Delta_1,
\end{equation*}
where $\psi_e \in X_h^1$ is defined in \eqref{me}.

\subsubsection{ The map $J_h^3: X_h^3 \rightarrow \C_0 \otimes \mathbb{V}$} 
For any $u \in X_h^3$, define
\begin{equation*}
    \langle L y^*, J_h^3 u \rangle =  \langle u, L \lambda_y \rangle    \qquad \text{ for } y \in \Delta_0, L \in \mathbb{V},
\end{equation*}
where  $\lambda_y$ is the continuous piecewise linear  function defined on the mesh $\Th$ associated to the vertex $y$ (e.g. $\lambda_y(z)=0$ for $z \in \Delta_0$, $z \neq y$, $\lambda_y(y)=1$). 

\subsection{The horizontal maps $\mathsf{r}_k$}
From this we can define naturally, $\mathsf{r}_0, \mathsf{r}_1, \mathsf{r}_2$ by requiring the above diagram commutes. 

\subsection{Definition of $\mathsf{r}_0$}
We define $\mathsf{r}_0$ by 
\begin{equation*}
    \mathsf{r}_0 J_h^0 u = J_h^1\varepsilon  u  \qquad  \text{ for all } u \in X_h^0.
\end{equation*}
In particular, if $u= L \lambda_y$, where $L \in \mathbb{V}$ and $y \in \Delta_0$, then $J_h^0 u= L y^*$. On the other hand,  if $e =[x,z]$  we have 
\begin{equation*}
   \langle J_h^1(\varepsilon  u),  e \rangle = \int_e t_e^T \varepsilon (L \lambda_y) t_e = L \cdot t_e \big(\lambda_y(z)-\lambda_y(x)\big).  
\end{equation*}

Hence, we see that 
\begin{equation*}
    \langle \mathsf{r}_0 (L y^*), e \rangle = L \cdot t_e \big( \lambda_y(z)-\lambda_y(x)\big).%= L \cdot t_e \mathsf{d}^0 y^*(e).
\end{equation*}
Or, equivalently, 
\begin{equation*}
    \mathsf{r}_0 (L y^*)= \sum_{e=[x ,y] \in \Delta_1} (L \cdot t_e) e^*. 
 \end{equation*}   

\subsection{Definition of $\mathsf{r}_1$}
We define $\mathsf{r}_1$ by 
\begin{equation*}
    \mathsf{r}_1 J_h^1 u = J_h^2 \inc  u  \qquad  \text{ for all } u \in X_h^1.
\end{equation*}
In particular, if $u= \psi_e$ then $J_h^1 u= e^*$. On the other hand,   we have 
\begin{equation*}
   \langle g^*, J_h^2 \inc  u \rangle = \langle \inc u, \psi_g \rangle =   \sum_{f \in \Delta_1} \llbracket \psi_e \rrbracket_f \langle t_f t_f^T \delta_f, \psi_g \rangle= \llbracket \psi_e \rrbracket_g 
\end{equation*}

Thus, 
\begin{equation*}
   \langle  g^*,  \mathsf{r}_1 e^* \rangle= \llbracket \psi_e \rrbracket_g. 
\end{equation*}
Let us define the collection of tetrahedra such that we have $e$ as an edge as follows: $\mathfrak{Z}(e)= \{ T \in \Delta_3: e \in \Delta_1(T) \}$. Now we let $\mathfrak{B}(e)=\{ g \in \Delta_1: g \in \Delta_1(T), \text{ for at least one } T \in \mathfrak{Z}(e) \}$ be the collection of edges of those tetrahdra.

Equivalently, we can write
\begin{equation*}
\mathsf{r}_1 e^*= \sum_{g \in \mathfrak{B}(e)}  \llbracket \psi_e \rrbracket_g \, g .
\end{equation*}

\subsection{Definition of $\mathsf{r}_2$}
We define $\mathsf{r}_2$ by 
\begin{equation*}
    \mathsf{r}_2 J_h^2 u = J_h^3 \dive  u  \qquad  \text{ for all } u \in X_h^2.
\end{equation*}
In particular, if $u= t_e t_e^T \delta_e $ for $e=[x,z] \in \Delta_1$ then $J_h^2 u= e$. On the other hand,  we have 
\begin{equation*}
   \langle  L y^*, J_h^3 \dive u \rangle = \langle \dive u , L \lambda_y \rangle =    \langle \delta_z-\delta_x, t_e \cdot L \lambda_y  \rangle= t_e \cdot L (\lambda_y(z)-\lambda_y(x)). 
\end{equation*}
Thus, 
\begin{equation*}
    \mathsf{r}_2 e=  t_e (z-x). 
\end{equation*}

We collect all the horizontal maps here:

\begin{alignat*}{2}
\mathsf{r}_0 (L y^*)= & \sum_{e=[x ,y] \in \Delta_1} (L \cdot t_e) e^*  \qquad && y \in \Delta_0, L \in \mathbb{V}, \\
\mathsf{r}_1 e^*= &\sum_{g \in \mathfrak{B}(e)}  \llbracket \psi_e \rrbracket_g \, g, \qquad && e \in \Delta_1 \\
\mathsf{r}_2 e= & t_e (z-x)  \qquad && e \in \Delta_1.
\end{alignat*}

We see that the maps $\mathsf{r}_0$ and $\mathsf{r}_2$ are quite simple. The map $\mathsf{r}_1$ is quite involved. Let $e=[x,z]$, then Christiansen showed that 
\begin{alignat*}{1}
    \psi_e= \frac{-1}{2|e|}\Big(\grad \lambda_x  (\grad \lambda_z)^T+ \grad \lambda_z  (\grad \lambda_x)^T \Big).
\end{alignat*}
If $T$ is a tetrahedra that has $e$ as an edge, then $m_e$ restricted to $T$ can be easily written using the geometry of $T$. However, the term $\llbracket \psi_e \rrbracket_g$ will be quite involved to calculate for any $g \in   \mathfrak{B}(e)$.

\section{A commuting diagram for the elasticity complex}\label{commutingdiscrete}
We consider the three complexes:
\[
\begin{tikzcd}[column sep=3em, row sep=3em]
C^{\infty}(\Omega)\otimes \mathbb{V} \arrow[r, "\varepsilon"] \arrow[d, "I_h^0"] &
C^{\infty}(\Omega)\otimes \mathbb{S} \arrow[r, "\inc"] \arrow[d, "I_h^1"] &
C^{\infty}(\Omega)\otimes \mathbb{S} \arrow[r, "\dive"] \arrow[d, "I_h^2"] &
C^{\infty}(\Omega)\otimes \mathbb{V} \arrow[d, "I_h^3"] \\
X_h^0 \arrow[r, "\varepsilon"] \arrow[d, "J_h^0"] & X_h^1 \arrow[r, "\inc"] \arrow[d, "J_h^1"] &
X_h^2  \arrow[r, "\dive"] \arrow[d, "J_h^2"] & X_h^3 \arrow[d, "J_h^3"] \\
\C^0 \otimes \mathbb{V} \arrow[r, "\mathsf{r}_0"] &
\C^1 \arrow[r, "\mathsf{r}_1"] &
\C_1  \arrow[r, " \mathsf{r}_2"] &
\C_0 \otimes \mathbb{V} 
\end{tikzcd}
\]

We define  vertical maps by composition:
\begin{equation*}
 E_h^k= J_h^kI_h^k.   
\end{equation*}
Then, we have the following diagram 
\begin{equation}\label{discreteElasticity}
\begin{tikzcd}[column sep=3em, row sep=3em]
C^{\infty}(\Omega)\otimes \mathbb{V} \arrow[r, "\varepsilon"] \arrow[d, "E_h^0"] &
C^{\infty}(\Omega)\otimes \mathbb{S} \arrow[r, "\inc"] \arrow[d, "E_h^1"] &
C^{\infty}(\Omega)\otimes \mathbb{S} \arrow[r, "\dive"] \arrow[d, "E_h^2"] &
C^{\infty}(\Omega)\otimes \mathbb{V} \arrow[d, "E_h^3"] \\
\C^0 \otimes \mathbb{V} \arrow[r, "\mathsf{r}_0"] &
\C^1 \arrow[r, "\mathsf{r}_1"] &
\C_1  \arrow[r, " \mathsf{r}_2"] &
\C_0 \otimes \mathbb{V} 
\end{tikzcd}
\end{equation}

We can easily prove that this diagram commutes using the fact that diagrams \eqref{Icommute} and \eqref{Jcommute} commute. 

For example, 
\begin{equation}
\mathsf{r}_0 E_h^0= \mathsf{r}_0 J_h^0 I_h^0 = J_h^1 \varepsilon I_h^0= J_h^1 I_h^1 \varepsilon =E_h^1 \varepsilon.  
\end{equation}

\subsection{Writing the maps $E_h^k$ more explicitly}
\subsubsection{ The map $E_h^0: C^{\infty}(\Omega)\otimes \mathbb{V} \rightarrow \C^0 \otimes \mathbb{V}$} 
For any $u \in C^{\infty}(\Omega)\otimes \mathbb{V}$, define
\begin{equation*}
    \langle E_h^0 u,  L  y  \rangle= L \cdot u(y)    \qquad \text{ for } y \in \Delta_0, L \in \mathbb{V}.
\end{equation*}

\subsubsection{ The map $E_h^1: C^{\infty}(\Omega)\otimes \mathbb{S} \rightarrow \C^1$} 
For any $u \in C^{\infty}(\Omega)\otimes \mathbb{S}$, define
\begin{equation*}
     \langle E_h^1 u, e \rangle =  \int_e t_e^T u t_e  \qquad \text{ for } e \in \Delta_1.
\end{equation*}

\subsubsection{ The map $E_h^2: C^{\infty}(\Omega)\otimes \mathbb{S} \rightarrow \C_1$} 
For any $w \in C^{\infty}(\Omega)\otimes \mathbb{S}$, define
\begin{equation*}
    \langle e^*, E_h^2 w \rangle =  ( w , \psi_e)  \qquad \text{ for } e \in \Delta_1.
\end{equation*}

\subsubsection{ The map $E_h^3:  C^{\infty}(\Omega)\otimes \mathbb{V} \rightarrow \C_0$} For any $u \in C^{\infty}(\Omega)\otimes  \mathbb{V}$, define
\begin{equation*}
    \langle L y^*, E_h^3 u \rangle=  ( u, L \lambda_y)     \qquad \text{ for } y \in \Delta_0, L \in \mathbb{V}.
\end{equation*}

    \section{Acknowledgements} 
     We would like to thank the entire MSRI-UP 2026 staff. The faculty members Alexander Diaz-Lopez, Johnny Guzm\'an, and Maurice Fabien, and the graduate mentors Parneet Gill and Elisabeth Rubio.  Our primary research mentor was Johnny Guzm\'an. Finally, we would like to thank  Snorre Christiansen for offering critical feedback and clarifying results in his papers.

     This work was carried out during the six week  MSRI-UP 2026 summer program which was hosted by the SL Math institute. We would like to thank SL Math for their hospitality and generosity.  
 
\bibliographystyle{abbrv}
\bibliography{./references}

\end{document}